\documentclass[12pt,leqno]{article}
\usepackage{amsmath,amsthm,amssymb}
\usepackage[utf8]{inputenc}
\usepackage[hang,flushmargin]{footmisc} 
\usepackage{marvosym}
\newcommand\blfootnote[1]{%
  \begingroup
  \renewcommand\thefootnote{}\footnote{#1}%
  \addtocounter{footnote}{-1}%
  \endgroup
}

\newtheorem{theorem}{THEOREM}[section] 
\newtheorem{lemma}{Lemma}[section]

\newtheorem{remark}{REMARK}[section]

\numberwithin{equation}{section}

\newcommand{\dx}{\, \mathrm{d}x}

\newcommand{\weenull}{\overset{\substack{\circ\\ \vspace{-0.34cm}}}{W}^{\substack{1,1\\ \vspace{-0.64cm}}}}
\newcommand{\rn}{\mathbb{R}^{n}}
\newcommand{\RN}{\mathbb{R}^{N}}
\newcommand{\R}{\mathbb{R}}
\newcommand{\loc}{\mathrm{loc}}

\newcommand{\N}{{\mathbb{N}}}

\newcommand{\udel}{u_\delta}
\newcommand{\wudel}{\widetilde{u}_\delta}
\newcommand{\nabdel}{\nabla u_\delta}
\newcommand{\intom}{\int_\Omega}
\newcommand{\intomd}{\int_{\Omega-D}}
\newcommand{\gr}[1]{(\ref{#1})}

\def\IntO{\int_\Omega}

\def\IntQ2Rx{\int_{{Q_{2R} (x_0)}}}

\def\Jpe{\widetilde{J}}
\def\O{\Omega}

\addtocounter{section}{0}

\title{\bf On the local boundedness of generalized minimizers of variational problems with linear growth}
\author{M.~Bildhauer \and M.~Fuchs \and J.~M\"uller \and X.~Zhong}
\date{}

\begin{document}

\parindent2ex

\maketitle
\noindent \\
\begin{bf}AMS classification\end{bf}: 49N60, 49Q20, 49J45
\noindent \\
\begin{bf}Keywords\end{bf}: variational problems of linear growth, TV-regularization, denoising and inpainting of images, local boundedness of solutions.

\begin{abstract} We prove local boundedness of generalized solutions to a large class of variational problems of linear growth including boundary value problems of minimal surface type and models from image analysis related to the procedure of TV--regularization occurring in connection with the denoising of images, which might even be coupled with an inpainting process. Our main argument relies on a Moser--type iteration procedure.
\end{abstract}
\blfootnote{
\begin{large}\Letter\end{large}Michael Bildhauer (bibi@math.uni-sb.de), \\
Martin Fuchs (fuchs@math.uni-sb.de),\\
 Jan Müller (corresponding author) (jmueller@math.uni-sb.de),\\
Saarland University (Department of Mathematics), P.O. Box 15 11 50, 66041 Saarbrücken, Germany,\\
Xiao Zhong (xiao.x.zhong@helsinki.fi),\\
 FI-00014 University of Helsinki (Department of Mathematics), P.O. Box 68 (Gustaf Hällströmin katu 2b),  00100 Helsinki (Finland).
}

\section{Introduction}
In this note we investigate variational problems of linear growth defined for functions $u : \O \to \R^N$ on a domain $\O \subset \rn$. The general framework of these kind of problems is explained e.g. in the monographs \cite{Giu,GMS1,GMS2,AFP,Bi}, where the reader interested in the subject will find a lot of further references as well as the definitions of the underlying spaces such as $ \mbox{BV} (\O,\R^N)$ and
$W^{1,p} (\O,\R^N)$ (and their local variants) consisting of all functions having finite total variation and the mappings with first order distributional derivatives located in the Lebesgue class $L^p (\O,\R^N)$, respectively. We will mainly concentrate on the case $n \ge 2$ assuming that $\O$ is a bounded Lipschitz region, anyhow, the case $n = 1$ can be included but is accessible by much easier means as it is outlined for example in \cite{BGH} and \cite{FMT}. To begin with, we consider the minimization problem
\begin{equation}
\label{G1}
J [w] := \IntO F (\nabla w) \dx \to \min \ \mbox{in} \ u_0 + \weenull (\O,\R^N)
\end{equation}
with boundary datum
\begin{equation}
\label{G2}
u_0 \in W^{1,1} (\O,\R^N)\, ,
\end{equation}
where $ \weenull (\O,\R^N)$ is the class of all functions from the Sobolev space $W^{1,1} (\O,\R^N)$ having vanishing trace (see, e.g., \cite{Ad}). Throughout this note we will assume that the energy density $F : 
\R^{N\times n} \to [0, \infty)$ satisfies the following hypotheses: 
\begin{gather}
\label{G3} F \in C^2 (\R^{N\times n})\text{ and  (w.l.o.g.) } \ F (0) = 0;\\
\label{G5} \nu_1 |P| - \nu_2 \le F (P) \le \nu_3|P| + \nu_4;\\
\label{growth} 0\leq D^2F(P)(Q,Q)\leq \nu_5\frac{1}{1+|P|}|Q|^2.
\end{gather}
with suitable constants $\nu_1, \nu_3, \nu_5 > 0, \ \nu_2, \nu_4 \ge 0$ and for all $P,Q\in\R^{N\times n}$. For notational simplicity, we collect the constants $\nu_i$ in a tuple
\[
\nu:=\big(\nu_1,...,\nu_5\big). 
\]
\begin{remark}\label{rm1.1}
 We note  that the above assumptions on $F$ particularly imply 
\begin{equation}
\label{G4}
|DF (P)| \le c(n)\cdot\max\{\nu_1,\nu_3\},
\end{equation}
which is a consequence of the linear growth condition \gr{G5} combined with the fact that $F$ is a convex function, which follows from the first inequality in \gr{growth}. A short proof of estimate \gr{G4} is given in \cite{Da}, Lemma 2.2 on p. 156. Moreover, the convexity of $F$ together with \gr{G5} also yields
\[
0 = F (0) \ge F (P) - P:D F (P) \ge \nu_1 |P| - \nu_2 - P : DF (P),
\]
hence
\begin{equation}
\label{G6}
DF (P) : P \ge \nu_1 |P| - \nu_2\, , \ P \in \R^{N\times n}.
\end{equation}
\end{remark}
As a matter of fact, problem (\ref{G1}) has to be replaced by its relaxed variant (see, e.g., \cite{AFP}, p. 303 and Theorem 5.47 on p. 304, or \cite{Bi}, chapter 4, as well as \cite{Bi2})
\begin{eqnarray}
\label{G7}
\Jpe [w] &: =& \IntO F \left(\nabla^a w\right) \dx + \IntO F^\infty \left(\frac{\nabla^s w}{|\nabla^s w|}\right) \mathrm{d} \left|\nabla^s w\right| \\[2ex]
&&+ \int_{\partial \O} F^\infty \left(\left(u_0 - w\right) \otimes\mathfrak{n}\right) \mathrm{d}  \mathcal{H}^{n - 1} \to \min  \mbox{ in } BV(\O,\R^N). \nonumber
\end{eqnarray}
Here $\nabla w = \nabla^a w \, {\cal{L}}^n + \nabla^s w$ is the Lebesgue decomposition of the measure $\nabla w$, $F^\infty$ is the recession function of $F$, i.e.
\[
F^\infty (P) = \lim_{t \to \infty} \frac{1}{t} \,F (tP), \ P \in \R^{N\times n},
\]
$\mathcal{H}^{n - 1}$ is Hausdorff's measure of dimension $n - 1$, and $\mathfrak{n}$ denotes the outward unit normal to $\partial \O$. By construction, problem (\ref{G7}) admits at least one solution, and the main result of \cite{Bi2} (compare Theorem 3 in this reference) states: 
\begin{theorem}
Let (\ref{G2}) - (\ref{growth}) hold together with $n = 2$ and $N=1$. Assume in addition that $F$ is of class $C^2$ satisfying for some $\mu > 1$ the condition of $\mu$-ellipticity
\begin{equation}
\label{G8}
\nu_6\left(1 + |P|\right)^{- \mu} |Q|^2 \le D^2 F (P) (Q, Q), \ P, Q \in \R^2\, ,
\end{equation}
with a constant $\nu_6 > 0$.
\begin{enumerate}
\item[a)] Assume $\mu \le 3$ in (\ref{G8}). Then (\ref{G7}) admits a solution $u^\ast$ in the space $W^{1,1} (\O)$. For each subdomain $\O^\ast \subset  \subset \O$ we have
\[
\int_{\O^\ast} \left|\nabla u^\ast \right| \ln (1 + |\nabla u^\ast|) \dx < \infty\, ,
\]
and any BV-solution $u$ of (\ref{G7}) differs from $u ^\ast$ by an additive constant.
\item[b)] If the case $\mu < 3$ is considered, then $u^\ast$ from a)  is actually of class $C^{1,\alpha} (\O)$ for any $\alpha \in (0, 1)$.
\end{enumerate}
\end{theorem}
\begin{remark}\label{rm1.2}
The above results extend to vector valued functions $u : \O \to \RN, N \ge 2$, provided we impose the structure condition
\begin{equation}
\label{G9}
F (P) = \widetilde{F} \big(|P|\big)
\end{equation}
for a suitable function $\widetilde{F} : [0, \infty) \to [0, \infty)$ of class $C^2$ which satisfies appropriate requirements implying \gr{G3}-\gr{growth} for $F$. For details we refer to the appendix.
\end{remark}
\begin{remark}
The main feature of Theorem 1.1 is that the ellipticity condition (\ref{G8}) together with an upper bound on the parameter $\mu$ is sufficient for obtaining a minimizer in a Sobolev class or even in a space or smooth functions. At the same time, the counterexample in section 4.4 of \cite{Bi} shows the sharpness of the limitation $\mu \le3$.
\end{remark}
\noindent Our first result concerns the situation where we drop condition (\ref{G8}) or allow values $\mu > 3$ even without restriction on the dimension $n$.
\begin{theorem}
Under the assumptions (\ref{G2}) - (\ref{growth}) the variational problem (\ref{G7}) has a solution $u \in BV(\Omega,\R^N)$, which in addition is a locally bounded function, i.e. $u\in BV(\Omega,\R^N)\cap L^\infty_\loc(\Omega,\R^N)$.
\end{theorem}

\begin{remark}
Note that we merely impose (\ref{G2}) on the boundary data. If we assume $u_0 \in L^\infty (\O,\R^N)$, then any solution $u$ of (\ref{G7}) is in the space $L^\infty (\O,\R^N)$,  which follows from the results in \cite{BF6}.
\end{remark}
\noindent Next, we look at a variational problem originating in the work of Rudin, Osher and Fatemi \cite{ROF} on the denoising of images. To be precise, we assume that $n=2$, $N=1$ and consider a measurable subset (``the inpainting region'') $D$ of $\O\subset\R^2$ such that
\begin{equation}
\label{G10}
0 \le {\cal {L}}^2 (D) < {\cal{L}}^2 (\O)\, ,
\end{equation}
where ${\cal{L}}^2 (D) = 0$ corresponds to the case of ``pure denoising''. Moreover, we consider given (noisy) data $f : \O - D \to \R$ such that
\begin{equation}
\label{G11}
f \in L^{2} (\O- D)
\end{equation}
and pass to the problem
\begin{equation}
\label{G12}
K [w] := \IntO F (\nabla w) \dx + \lambda \int_{\O - D} |w - f|^2 \dx \to  \min \ \mbox{in} \ W^{1,1} (\O) \, ,
\end{equation}
where $\lambda > 0$ is some parameter and $F$ satisfies (\ref{G3}) - (\ref{growth}). The problem (\ref{G12}) can be regarded as a model for the inpainting of images combined with simultaneous denoising. The relaxed version of (\ref{G12}) reads as 
\begin{eqnarray}
\label{G13}
\widetilde{K} [w] &: =& \IntO F \left(\nabla^a w\right) \dx + \IntO F^\infty \left(\frac{\nabla^s w}{|\nabla^s w|}\right) \mathrm{d} \left|\nabla^s w\right| \\[2ex]
&&+ \lambda \int_{\O - D}  |w - f|^2 \dx  \to \min \ \mbox{in} \ \mbox{BV} (\O) \, , \nonumber
\end{eqnarray}
and concerning the regularity of solutions of (\ref{G13}) we obtained in \cite{BFT}, Theorem 2:
\begin{theorem}
Consider a density $F$ as in Theorem 1.1 for which (\ref{G8}) holds with $\mu < 2$. Moreover, we replace (\ref{G11}) with the stronger condition $f \in L^\infty (\O - D)$. Then the problem (\ref{G13}) (and thereby  (\ref{G12})) admits a unique solution $u$ for which we have interior $C^{1, \alpha}$-regularity on the domain $\O$.
\end{theorem}
\begin{remark}
The result of Theorem 1.3 extends to domains $\O$ in $\rn$ with $n \ge 3$, where we might even include the vector case of functions $u : \O \to \RN$, provided we have (\ref{G9}) in case $N>1$. We refer to \cite{Ti}. The reader should also note that boundedness of the data $f$ implies the boundedness of solutions to (\ref{G13}) (see, e.g., \cite{BF2}).
\end{remark}
\begin{remark}
In the paper \cite{BFMT} the reader will find some intermediate regularity results for solutions $u$ of (\ref{G13}) saying that even without the assumption $f \in L^\infty (\O - D)$ the solution $u$ belongs to some Sobolev class. With respect to these results we can even replace the ``data term'' $ \int_{\O- D} |u - f|^2 \dx$ by more general expressions (with appropriate variants of (\ref{G11})), however, in any case $\mu$-ellipticity (\ref{G8}) together with an upper bound on $\mu$ is required.
\end{remark}
\begin{remark}
The counterexamples from \cite{FMT} show that for $\mu > 2$ we can not in general hope for the solvability of problem (\ref{G12}), which means that for these examples any solution $u$ of (\ref{G13}) belongs to $\mbox{BV} (\O) - W^{1,1} (\O)$.
\end{remark}
In the spirit of Theorem 1.2 we have the following weak regularity result for problem (\ref{G13}).
\begin{theorem}
Let (\ref{G3}) - (\ref{growth}) hold, let $D$ satisfy (\ref{G10}), suppose that $n=2$, $N=1$ and consider data $f$ with (\ref{G11}). Then problem (\ref{G13}) admits a solution $u$ in the space $\mbox{BV} (\O) \cap L^\infty_{\loc} (\O)$, which is unique in the case of pure denoising (i.e. $D = \emptyset$).
\end{theorem}
\noindent Our paper is organized as follows: in Section 2 we introduce a new type of linear regularization of the problems (\ref{G7}) and (\ref{G13}) by means of $\mu$-elliptic functionals including results on the regularity and the convergence properties of the family of approximate solutions $\udel$. In Section 3 we then derive local uniform bounds of the type
\begin{equation}
\label{G14}
\sup_{\delta > 0} \|u_\delta\|_{L^\infty (\O^\ast,\R^N)} \le c (\O^\ast) < \infty
\end{equation}
for subdomains $\O^\ast \subset \subset \O$ by a Moser-type iteration procedure, which yields the result of Theorem 1.2 by passing to the limit $\delta \downarrow 0$. In the last section we will deduce the statement of Theorem 1.4 from the proof of Theorem 1.2.

\section{$\mu$-elliptic regularization}
In the context of variational problems of linear growth, it is a common approach to consider a sequence of regularizing functionals whose minimizers are sufficiently smooth and converge to a solution of the actual problem. In our previous works (cf. e.g. \cite{BF1,BF2,BF3,FMT}) this  was achieved by adding a Dirichlet term $\delta\intom |\nabla w|^2\dx$ for a decreasing sequence $\delta\downarrow 0$. For fixed $\delta$, we then deal with a quadratic elliptic functional and therefore have the well developed machinery for this type of problems, as it is e.g. outlined in the classical monograph \cite{GT}. However, in the situation of Theorems 1.2 and 1.4, a quadratic regularization and the resulting inhomogeneity between the linear and the quadratic term causes some difficulties. We therefore prefer to work with a linear regularization, for which the notion of $\mu$-ellipticity (cf. \gr{G8}) turns out to be the correct framework in terms of existence and regularity of approximating solutions. Let us first consider the situation of Theorem 1.2, where, just for technical simplicity, we replace (\ref{G2}) by the requirement that 
\begin {equation}
\label{H1}
u_0 \in W^{1,p} (\O,\R^N)\, 
\end{equation}
for some $p>1$. We would like to note that the limit case $p=1$ can be included via a suitable approximation (cf. \cite{BF7} and in particular the work \cite{Bi2}, where the approximation is made explicit in the two-dimensional case). We may therefore actually drop (\ref{H1}) and return to the original hypothesis (\ref{G2}). Now for $0 < \delta < 1$ let
\begin{equation}
\label{H2}
J_\delta [w] :=\delta \IntO F_\mu(\nabla w)\dx + J [w] \to  \min \ \mbox{in} \ u_0 + \weenull(\O,\R^N),
\end{equation}
where $F_\mu:\R^{N\times n}\rightarrow [0,\infty)$ is chosen to satisfy
\begin{gather}
\label{G3'} F_\mu \in C^2 (\R^{N\times n})\text{ and  (w.l.o.g.) } \ F_\mu (0) = 0\, ;\\
\label{G5'} \widetilde{\nu}_1 |P| - \widetilde{\nu}_2 \le F_\mu (P) \le \widetilde{\nu}_3|P| + \widetilde{\nu}_4;\\
\label{growth'} \widetilde{\nu}_5\left(1 + |P|\right)^{- \mu} |Q|^2 \leq D^2F_\mu(P)(Q,Q)\leq \widetilde{\nu}_6\frac{1}{1+|P|}|Q|^2,
\end{gather}
with suitable constants $\widetilde{\nu}_1, \widetilde{\nu}_3, \widetilde{\nu}_5,\widetilde{\nu}_6 > 0, \ \widetilde{\nu}_2, \widetilde{\nu}_4 \ge 0$, some $\mu\in (1,\infty)$ and for all $P,Q\in\R^{N\times n}$. Again we set
\[
 \widetilde{\nu}:=\big(\widetilde{\nu}_1,...,\widetilde{\nu}_6\big).
\]
We further note that the above assumptions imply                                                                                                                                                                                                                                                                                                                                                                                                                      \begin{align}\label{muconst}
DF_\mu(P):P\geq \widetilde{\nu}_1|P|-\widetilde{\nu}_2,\quad P\in \R^{N\times n}.
\end{align}
If the vector case $N>1$ is considered, we impose a structure condition on $F_\mu$ in the spirit of \gr{G9}, i.e.
\[
F_\mu(P)=\widetilde{F}_\mu\big(|P|\big)
\]
for some $\mu$-elliptic function $\widetilde{F}_\mu:\R\to\R$, implying the above assumptions for $F_\mu$ (compare the appendix).
A convenient choice for $F_\mu$ is e.g. given by
\[
F_\mu(P)=\Phi_\mu\big(|P|\big),
\]
where  $\Phi_\mu$ is defined by
\begin{align*}
\Phi_\mu(r):=\int_0^r\int_0^s(1+t)^{-\mu}\,dt\,ds,\;r\geq 0,
\end{align*}
which means
\begin{align*}
\left\{\begin{aligned}
 &\Phi_\mu(r)=\frac{1}{\mu-1}r+\frac{1}{\mu-1}\frac{1}{\mu-2}(r+1)^{-\mu+2}-\frac{1}{\mu-1}\frac{1}{\mu-2},\;\mu\neq 2,\\
 \,\\
 &\Phi_2(r)=r-\ln(1+r),\;r\geq 0.
\end{aligned}\right.
\end{align*}
\begin{lemma}\label{lem2.1}
If we fix $1<\mu<1+\frac{2}{n}$, then we have:
\begin{enumerate}
\item[a)] Problem \gr{H2} admits a unique  solution $\udel\in u_0+\weenull(\Omega,\R^N)$. It even holds (not necessarily uniformly with respect to $\delta$) $\udel\in C^{1,\alpha}(\Omega,\R^N)$.
\item[b)] $\displaystyle\sup_\delta \IntO |\nabla u_\delta| \dx < \infty$\, ;
\item[c)] $\displaystyle\IntO D F_{\delta,\mu} (\nabla u_\delta) \cdot \nabla \varphi \dx = 0 $ for any  $\varphi \in \weenull(\O,\R^N), \ F_{\delta,\mu} (P) :=\delta F_\mu(P)+ F (P)$\, .
\item[d)] Each $L^1$-cluster point of the family $u_\delta$ is a solution of problem (\ref{G7}).
\end{enumerate}
\end{lemma}
\begin{proof}
It is easy to see that under our assumptions on $F$ the density $F_{\delta,\mu}$ is $\mu$-elliptic itself in the sense of \gr{G8}, so that we may cite the results from \cite{BF8} for part a). Part b) and c) are clear from the fact that $u_\delta$ minimizes $J_\delta$. For part c) we observe that due to b) and the $BV$-compactness property (see Theorem 3.23 on p. 132 in \cite{AFP}), there exists a function $\overline{u}\in BV(\Omega,\R^N)$ such that $u_\delta\rightarrow \overline{u}$ in $L^1(\Omega)$ for some sequence $\delta\downarrow 0$. Thanks to the lower semicontinuity of the functional $\Jpe$ from \gr{G7}, it follows
\[
\Jpe[\overline{u}]\leq \liminf_{\delta\rightarrow 0}\Jpe[\udel]=\liminf_{\delta\rightarrow 0}J[\udel]\leq \liminf_{\delta\rightarrow 0}J_\delta[\udel]\leq \liminf_{\delta\rightarrow 0}J_\delta[v]=J[v],
\]
where $v\in u_0+\weenull(\Omega,\R^N)$ is arbitrary. But since in \cite{BF9} it was proved that the set of $\Jpe$-minimizers coincides with the set of all $L^1$-limits of $J$-minimizing sequences, the above chain of inequalities implies the claimed minimality.
\end{proof}

Next we consider the setting of Theorem 1.4. Keep in mind that in this situation we restrict ourselves to $n=2$ and $N=1$. Since we merely assume $f\in L^2(\Omega-D)$, we need to ``cut-off'' the data in order to obtain a sufficiently smooth approximation. This means that for $\delta\in (0,1)$ we set
\[
f_\delta:\Omega-D\rightarrow\R,\,f_\delta(x):=\left\{\begin{aligned}f(x),&\,\text{ if }\,|f(x)|\leq \delta^{-1}, \\ \delta^{-1},&\,\text{ if }\,|f(x)|> \delta^{-1}\end{aligned}\right.
\]
and consider the problem
\begin{align}\label{Kdel}
\begin{split}
K_\delta[w]:=\delta\intom F_\mu(\nabla w)\dx+ \IntO F (\nabla w) \dx + \lambda \int_{\O - D} &|w - f_\delta|^2 \dx\\
 &\to\min \text{ in }W^{1,1}(\Omega).
\end{split}
\end{align}

\begin{lemma}
If we fix $1<\mu<2$, then we have:
\begin{enumerate}
\item[a)] Problem \gr{Kdel} admits a unique  solution $\wudel\in W^{1,1}(\Omega,\R^N)$. It even holds (not necessarily uniformly with respect to $\delta$) $\wudel\in C^{1,\alpha}(\Omega,\R^N)$.
\item[b)] $\displaystyle\sup_\delta \IntO |\nabla \wudel| \dx < \infty$,\; $\displaystyle\sup_\delta \intomd |\wudel|^2 \dx < \infty$ ;
\item[c)] $\displaystyle\IntO D F_{\delta,\mu} (\nabla \wudel) \cdot \nabla \varphi \dx+\lambda\intomd (\wudel-f_\delta)\varphi\dx = 0$ for any $\varphi \in W^{1,1}(\O,\R^N)$,  

$F_{\delta,\mu} (p) :=\delta F_\mu(p)+ F (p)$.
\item[d)] Each $L^1$-cluster point of the sequence $\wudel$ is a solution of problem (\ref{G7}).
\end{enumerate}
\end{lemma}

\begin{proof}
Since  $f_\delta\in L^\infty(\Omega)$ for each fixed value of $\delta$, we are in the situation of \cite{BFT}, where we remark that the  density $F_{\delta,\mu}$ is $\mu$-elliptic  thanks to our assumptions on $F$. We can therefore apply the results of this work which give us the claim of part a). Parts b) and c) are once again clear from the minimality of the $\wudel$, where for the second bound in b) we have to make use of the fact that $f_\delta\to f$ in $L^2(\Omega-D)$. It thus remains to justify d). By the bounds of part b), the family $\wudel$ is bounded uniformly in $W^{1,1}(\Omega)$ and hence there exists an $L^1$-cluster point $\hat{u}\in BV(\Omega)$ of some sequence $\delta\downarrow 0$ due to the $BV$-compactness property. From the lower semicontinuity of the relaxation $\widetilde{K}$ it then follows for arbitrary $v\in W^{1,1}(\Omega)$
\begin{align*}
 \widetilde{K}[\hat{u}]&\leq \liminf_{\delta\downarrow 0}\widetilde{K}[\wudel]\leq \liminf_{\delta\downarrow 0} \left[K_\delta[\wudel]+\lambda\intomd \Big(|\wudel-f|^2-|\wudel-f_\delta|^2\Big)\dx\right]\\
 &\leq \liminf_{\delta\downarrow 0}\left[K_\delta[v]+\lambda\intomd \Big(|\wudel-f|^2-|\wudel-f_\delta|^2\Big)\dx\right]\\
&=K[v]+\liminf_{\delta\downarrow 0}\lambda\intomd \Big(|\wudel-f|^2-|\wudel-f_\delta|^2\Big)\dx\\
&=K[v]+\liminf_{\delta\downarrow 0}\intomd \Big(|f|^2-|f_\delta|^2+2\wudel(f_\delta-f)\Big)\dx=K[v],
\end{align*}
since $f_\delta\rightarrow f$ in $L^2(\Omega-D)$ and $\wudel$ is uniformly bounded in $L^2(\Omega-D)$. The claimed minimality of $\hat{u}$ now follows from the fact that any function $w\in BV(\Omega)$ can be approximated by a sequence $w_k\in C^\infty(\Omega)\cap W^{1,1}(\Omega)$ such that $\widetilde{K}[w]=\lim_{k\rightarrow\infty}\widetilde{K}[w_k]$ (cf. Lemma 2.1 and 2.2 in \cite{FT}).
\end{proof}

\section{Proof of Theorem 1.2}
We consider the general case $n\geq 2$, $N\geq 1$. Our starting point is the Euler equation from Lemma 2.1 c)
\begin{align}\label{Eeq}
\delta \intom DF_\mu(\nabdel):\nabla\varphi\dx+\intom DF(\nabdel):\nabla\varphi\dx=0,
\end{align}
where we choose $\varphi=\eta^2|\udel|^s\udel$ for some positive exponent $s$ and a function $\eta\in C^1_0(\Omega)$, $0\leq\eta\leq 1$, which is an admissible choice due to Lemma 2.1 a).  We observe
\[
\nabla \varphi=\udel\otimes\Big(2\eta |\udel|^s\nabla\eta+\eta^2\nabla\big(|\udel|^s\big)\Big)+\eta^2|\udel|^s\nabdel
\]
and therefore
\begin{align}\label{3.9}
\begin{split}
DF&(\nabdel):\nabla \varphi\\
=&2\eta |\udel|^s DF(\nabdel):(\udel\otimes\nabla\eta)+\eta^2DF(\nabdel):\Big(\udel\otimes\nabla\big(|\udel|^s\big)\Big)\\
&\hspace{5.5cm}+\eta^2|\udel|^s DF(\nabdel):\nabdel=:T_1+T_2+T_3.
\end{split}
\end{align}
Note that due to \gr{G6} we have

\[
T_3= \eta^2|\udel|^s DF(\nabdel):\nabdel\geq \nu_1\eta^2|\nabdel||\udel|^s-\nu_2\eta^2|\udel|^s.
\]
For the term $T_2$ of \gr{3.9}, we use the structure condition \gr{G9} (in case $N>1$) and get
\begin{align*}
DF(\nabdel):\Big(\eta^2\udel\otimes\nabla\big(|\udel|^s\big)\Big)=\frac{\widetilde{F}'\big(|\nabdel|\big)}{|\nabdel|}\Big(\eta^2\udel\otimes\nabla \big(|\udel|^s\big)\Big):\nabdel.
\end{align*}
From
\begin{align*}
\Big(\eta^2\udel\otimes\nabla \big(|\udel|^s\big)\Big):\nabdel&=\frac{1}{2}s|\udel|^{s-1}\eta^2\nabla|\udel|\cdot \nabla|\udel|^2\\
&=s|\udel|^{s}\eta^2\nabla|\udel|\cdot \nabla|\udel|\geq 0
\end{align*}
we then obtain the estimate
\begin{align*}
 DF(\nabdel):\nabla \varphi\geq 2\eta |\udel|^s DF(\nabdel):(\udel\otimes\nabla\eta) +\nu_1\eta^2|\nabdel||\udel|^s-\nu_2\eta^2|\udel|^s
\end{align*}
and similarly (compare the definition of $F_\mu$ and recall inequality \gr{muconst})
\begin{align*}
 DF_\mu(\nabdel):\nabla \varphi\geq 2\eta |\udel|^s DF_\mu(\nabdel):(\udel\otimes\nabla\eta) +\widetilde{\nu}_1\eta^2|\nabdel||\udel|^s-\widetilde{\nu}_2\eta^2|\udel|^s.
\end{align*}
Note that in the scalar case these inequalities are valid without condition \gr{G9}.
The Euler equation (\ref{Eeq}) then implies (using the boundedness of $DF$ and  $DF_\mu$, compare \gr{G4})
\begin{align}\label{3.14}
\begin{split}
\intom |\nabdel|&|\udel|^s\eta^2\dx\leq c \left[\intom \eta^2|\udel|^s\dx+\intom |\udel|^{s+1}\eta|\nabla\eta|\dx\right]
\end{split}
\end{align}
for some constant $c=c(\nu,\widetilde{\nu})$.
In the next step we set
\[
v:=|\udel|^{s+1}\eta^2.
\]
Then 
\[
|\nabla v|\leq (s+1)\eta^2|\udel|^s\big|\nabla\big(|\udel|\big)\big|+2|\udel|^{s+1}\eta|\nabla\eta|\leq c(n)(s+1)\eta^2|\udel|^s|\nabdel|+2|\udel|^{s+1}\eta|\nabla\eta|.
\]
 Furthermore, from the Sobolev-Poincar\'e inequality we have
\[
\intom |\nabla v|\dx\geq c(n)\left(\intom |v|^\frac{n}{n-1}\dx\right)^\frac{n-1}{n},
\]
and we can therefore estimate the left-hand side of \gr{3.14} from below by
\begin{align*}
\intom |\udel|^s|\nabdel|\eta^2\dx\geq \frac{c(n)}{s+1}\left[\left(\intom |\udel|^{(s+1)\frac{n}{n-1}}\eta^\frac{2n}{n-1}\dx\right)^\frac{n-1}{n}-2\intom |\udel|^{s+1}\eta|\nabla\eta|\dx\right].
\end{align*}
We insert this into inequality \gr{3.14} which then yields 
\begin{align}\label{3.15}
\left(\intom |\udel|^{(s+1)\frac{n}{n-1}}\eta^\frac{2n}{n-1}\dx\right)^\frac{n-1}{n}\leq c(s+1)\left[\intom |\udel|^s\eta^2\dx+\intom |\udel|^{s+1}\eta |\nabla\eta|\dx\right]
\end{align}
with a constant $c=c(n,\nu,\widetilde{\nu})$. Now we fix some open ball $B_{R_0}$ inside $\Omega$. For any $j\in \N_0$ we set
\[
R_j:=\frac{n-1}{n}R_0+\Big(\frac{n-1}{n}\Big)^j\frac{R_0}{n}
\]
and consider the sequence of concentric open Balls $B_j$ of radius $R_j$ inside $B_0=B_{R_0}$. Note that
\[
\bigcap_{j=0}^\infty B_j\supset B_{\frac{n-1}{n}R_0}=:B_\infty.
\]
We further choose smooth functions $\eta_j\in C_0^\infty(B_j)$ such that $\eta_j\equiv 1$ on $B_{j+1}$, $0\leq\eta\leq 1$ and 
\[
|\nabla\eta_j|\leq \frac{2}{R_j-R_{j+1}}=c(R_0,n)\Big(\frac{n}{n-1}\Big)^j.
\]
Then, together with the choice $s_j:=\big(\frac{n}{n-1}\big)^j-1$, the inequality \gr{3.15} implies
\begin{align}\label{3.16}
\left(\int_{B_{j+1}}|\udel|^{(\frac{n}{n-1})^{j+1}}\dx\right)^\frac{n-1}{n}\leq c\Big(\frac{n}{n-1}\Big)^{2j}\left[\int_{B_j}|\udel|^{s_j}\dx+\int_{B_j}|\udel|^{s_j+1}\dx\right],\quad\forall j\in\N,
\end{align}
with a constant $c=c(n,\nu,\widetilde{\nu},R_0)$. 

In the following, we fix the value of the parameter $\delta\in (0,1)$ and note that by Hölder's inequality we have
\begin{align}\label{3.6}
\begin{split}
&\int_{B_j}|\udel|^{s_j}\dx\leq \left(\int_{B_j}|\udel|^{s_j+1}\dx\right)^\frac{s_j}{s_j+1}\cdot \left(\int_{B_j}1\dx\right)^\frac{1}{s_j+1}\\
&\leq c(R_0,n)\left(\int_{B_j}|\udel|^{s_j+1}\dx\right)^\frac{s_j}{s_j+1}.
\end{split}
\end{align}
Next we let 
\[
 a_j:=\max\left\{1,\int_{B_j}|\udel|^{\big(\frac{n}{n-1}\big)^j}\dx\right\}
\]
and obtain from \gr{3.16}
\begin{align*}
 \big(a_{j+1}\big)^{\frac{n-1}{n}}&\leq \max\left\{1,c \Big(\frac{n}{n-1}\Big)^{2j}\bigg[\int_{B_j}|\udel|^{s_j}\dx+\int_{B_j}|\udel|^{s_j+1}\dx\bigg]\right\}\\
&\leq  c \Big(\frac{n}{n-1}\Big)^{2j} \max\left\{1,\int_{B_j}|\udel|^{s_j}\dx+\int_{B_j}|\udel|^{s_j+1}\dx\right\}.
\end{align*}
On the right-hand side we apply inequality \gr{3.6} with the result 
\[
 \big(a_{j+1}\big)^{\frac{n-1}{n}}\leq c \Big(\frac{n}{n-1}\Big)^{2j} \max\left\{1,\int_{B_j}|\udel|^{s_j+1}\dx+\bigg(\int_{B_j}|\udel|^{s_j+1}\dx\bigg)^{\frac{s_j}{s_j+1}}\right\}
\]
for a suitable positive constant $c=c(n,\nu,\widetilde{\nu},R_0)$, hence we arrive at

\begin{align}\label{3.7}
\big(a_{j+1}\big)^\frac{n-1}{n}\leq c\Big(\frac{n}{n-1}\Big)^{2j}\cdot a_j\quad\forall j\in\N.
\end{align}
Through an iteration, we  obtain from \gr{3.7}
\begin{align}\label{3.8}
\begin{split}
&\|\udel\|_{L^{s_j+1}(B_j,\R^N)}\\
&\leq \big(a_j\big)^{\big(\frac{n-1}{n}\big)^j}\leq c^{\;\sum\limits_{k=1}^{j-1}\big(\frac{n-1}{n}\big)^{k}}\Big(\frac{n}{n-1}\Big)^{\;\sum\limits_{k=1}^{j-1}2k\big(\frac{n-1}{n}\big)^{k}}\cdot \max\Big\{1,\|\udel\|_{L^{\frac{n}{n-1}}(\Omega,\R^N)}\Big\},
\end{split}
\end{align}
and since 
\[
\sum\limits_{k=1}^\infty \Big(\frac{n-1}{n}\Big)^k=n-1 \quad\text{as well as}\quad \sum\limits_{k=1}^\infty 2k\Big(\frac{n-1}{n}\Big)^k=2n(n-1), 
\]
we may pass to the limit $j\rightarrow\infty$ which yields
\begin{align*}
\sup_{x\in B_\infty}|\udel(x)|=\lim_{j\rightarrow\infty}\|\udel\|_{L^{s_j+1}(B_\infty,\R^N)}\leq c^{n-1}\Big(\frac{n}{n-1}\Big)^{2n(n-1)}\cdot \max\Big\{1,\|\udel\|_{L^{\frac{n}{n-1}}(\Omega,\R^N)}\Big\}
\end{align*}
with the right-hand side being bounded independently of the parameter $\delta$ since due to Lemma 2.1 b), the sequence $\udel$ is uniformly bounded in $W^{1,1}(\Omega,\R^N)$ and hence by Sobolev's embedding  in $L^\frac{n}{n-1}(\Omega,\R^N)$.
The conclusion is that we find $\sup_{x\in B_\infty}|\udel(x)|$ to be bounded by some constant which does not depend on the parameter $\delta$, which means that also the $L^1$-limit $\overline{u}$ of the $\udel$ is locally bounded. This finishes the proof of Theorem 1.2.\qed

\section{Proof of Theorem 1.4}
We would like to remind the reader of the fact that in the setting of Theorem 1.4 we restrict ourselves to the case $n=2$ and $N=1$. So let $\wudel$ denote the solution from Lemma 2.2 a) and assume henceforth that we are in the situation of Theorem 1.4. Let $x_0\in \O$. We choose $R_0>0$ small enough such that $B_{R_0}(x_0)\subset \O$ and 
\begin{equation}
\label{r0}
\int_{B_{R_0}(x_0)-D}\vert f\vert^2\dx<\varepsilon_0. 
\end{equation}
Here $\varepsilon_0>0$ is small and will be determined soon. Let $\eta\in C^\infty_0(\O)$ be a non-negative cut-off function with support in $B_{R_0}(x_0)$ and $s\ge 0$ a non-negative number. By Lemma 2.2 a), we can use the following function 
\[ \varphi= \vert \wudel\vert^s\wudel\eta^2\]
as a testing function to the Euler equation in Lemma 2.2 c), and we obtain that
\begin{equation}\label{start1} 
\delta\intom DF_\mu(\nabla\wudel)\cdot\nabla \varphi\dx+\intom DF(\nabla \wudel)
\cdot\nabla \varphi\dx+\lambda\intomd (\wudel-f)\varphi\dx=0.
\end{equation}
Note that 
\[ \nabla \varphi=(s+1)\vert \wudel\vert^s \eta^2\nabla \wudel +2\vert \wudel\vert^s\wudel\eta\nabla\eta.\]
Thus by (\ref{G4}), (\ref{G6}) and \gr{muconst}, we have
\begin{equation}\label{eq1}
\begin{aligned}
DF_{\delta,\mu}&(\nabla \wudel)\cdot \nabla \varphi\ge \\
& \delta (s+1)\widetilde{\nu}_1 \vert \wudel\vert^s\vert \nabla \wudel\vert\eta^2-\delta(s+1)\widetilde{\nu}_2\vert \wudel\vert^s\eta^2-
2\delta|DF_\mu|\vert \wudel\vert^{s+1}\eta\vert\nabla\eta\vert\\
&+(s+1)\nu_1 \vert \wudel\vert^s\vert \nabla \wudel\vert\eta^2-(s+1)\nu_2\vert \wudel\vert^s\eta^2-
2|DF|\vert \wudel\vert^{s+1}\eta\vert\nabla\eta\vert.
\end{aligned}
\end{equation}
We also note that 
\begin{equation}\label{eq2}
2\lambda (\wudel-f)\wudel\vert \wudel\vert^s\eta^2\ge -2\lambda f\wudel\vert \wudel\vert^s\eta^2\ge 
-2\lambda \vert f\vert \vert \wudel\vert^{s+1}\eta^2.
\end{equation}
Now it follows from (\ref{start1}), (\ref{eq1}) and (\ref{eq2}) that
\begin{equation}\label{start2}
\begin{aligned}
&(\delta+1)(s+1)\int_\O  \vert \wudel\vert^s\vert\nabla \wudel\vert^2\eta^2\dx \\
 &\le c(\delta+1)(s+1)\left[ 
 \intom \vert \wudel\vert^s\eta^2\dx+\intom \vert \wudel\vert^{s+1}\eta\vert\nabla \eta\vert\dx\right]
 +2\lambda \intomd \vert f\vert \vert \wudel\vert^{s+1}\eta^2\dx
\end{aligned}
\end{equation}
with a constant $c=c(\nu,\widetilde{\nu})$.
As in the proof of Theorem 1.2, we let 
\[ v=\vert \wudel\vert^{s+1}\eta^2.\]
Then 
\[ \vert \nabla v\vert \le (s+1)\vert \wudel\vert^s\vert \nabla \wudel\vert\eta^2+2\vert \wudel\vert^{s+1}\eta\vert\nabla\eta\vert\] 
and by the Sobolev inequality, we further have
\[ c(n)\left(\intom v^2\dx\right)^{\frac{1}{2}}\le \intom \vert\nabla v\vert\dx.\]
Thus (\ref{start2}) implies
\begin{equation}\label{start4}
\begin{aligned}
(\delta+1) \left(\intom|v|^2 \dx\right)^{\frac{1}{2}} \le c(\delta+1)(s+1)\left[ 
 \intom \vert \wudel\vert^s\eta^2\dx+\intom \vert \wudel\vert^{s+1}\eta\vert\nabla \eta\vert\dx\right]\\
 +\underset{\mbox{$=:T$}}{\underbrace{2\lambda \intomd \vert f\vert \vert \wudel\vert^{s+1}\eta^2\dx}},
 \end{aligned}
\end{equation}
with a constant $c=c(n,\nu,\widetilde{\nu})$.
We will estimate the term $T$ in the following way: by the Hölder inequality and (\ref{r0}), 
\begin{equation*}
\begin{aligned}
2\lambda \intomd \vert f\vert\vert \wudel\vert^{s+1}\eta^2\dx&\le
2\lambda \left( \int_{B_{R_0}(x_0)-D} \vert f\vert^2\dx\right)^{\frac{1}{2}}\left(\intom \vert \wudel\vert^{2(s+1)}\eta^4\dx\right)^{\frac{1}{2}}\\
&\le 2\lambda \varepsilon_0^{1/2}\left(\intom \vert \wudel\vert^{2(s+1)}\eta^4\dx\right)^{\frac{1}{2}}.
\end{aligned}
\end{equation*}
If we choose $\varepsilon_0$ small such that
\[ 2\lambda \varepsilon_0^{1/2}\leq\frac{1}{2}<\frac{\delta+1}{2},\]
then the term $T$ can be absorbed in the left-hand side of (\ref{start4}), and we deduce from this inequality
\begin{equation}\label{start5}
\begin{aligned}
\left(\intom |\wudel|^{2(s+1)}\eta^4\dx\right)^{\frac{1}{2}} \le 2c(s+1)\left[ 
 \intom \vert \wudel\vert^s\eta^2\dx+\intom \vert \wudel\vert^{s+1}\eta\vert\nabla \eta\vert\dx\right].
\end{aligned}
\end{equation}
Note that this is just inequality \gr{3.15} (with $n=2$) from the preceding section, so that from this point on we can simply repeat the arguments which were used to obtain the uniform local boundedness of $\udel$ in the proof of Theorem 1.2. This finishes our proof. \qed 

\vspace{1cm}
\noindent\begin{Large}\textbf{Appendix: discussion of the structure condition \gr{G9}}\end{Large}
\vspace{0cm}

\noindent For the interested reader we explain Remark \ref{rm1.2} in a more detailed form.
\setcounter{section}{1}
\renewcommand{\thesection}{\Alph{section}}
\setcounter{theorem}{1}
\begin{lemma}
 Consider a function $\widetilde{F}:[0,\infty)\rightarrow [0,\infty)$ of class $C^2$ satisfying (with constants $\nu_1,\nu_3,\nu_5>0$, $\nu_2,\nu_4\geq 0$)
\begin{enumerate}
 \item[(A1)]$\widetilde{F}(0)=0$,
 \item[(A2)]$\widetilde{F}'(0)=0$,
 \item[(A3)]$\nu_1t-\nu_2\leq \widetilde{F}(t)\leq \nu_3t+\nu_4$,
 \item[(A4)]$\widetilde{F}''(t)\geq 0$,
 \item[(A5)]$\displaystyle\widetilde{F}''(t)\leq \nu_5\frac{1}{1+t}$
\end{enumerate}
for any $t\geq 0$. Then we have \gr{G3}-\gr{growth} for the density $F(P):=\widetilde{F}\big(|P|\big)$, $P\in\R^{N\times n}$. If in addition for some $\nu_6>0$ and $\mu>1$
\begin{enumerate}
 \item[(A6)]$\displaystyle\min\limits_{t\geq 0}\bigg\{\frac{\widetilde{F}'(t)}{t},\widetilde{F}''(t)\bigg\}\geq \nu_6 (1+t)^{-\mu}$,
\end{enumerate}
we obtain the condition \gr{G8} of $\mu$-ellipticity for $F$.
\end{lemma}
\begin{remark}
 The hypotheses (A1-6) hold for the function $\widetilde{F}(t):=\Phi_\mu(t)$ defined before Lemma \ref{lem2.1}.
\end{remark}

\noindent \textit{Proof of Lemma A.1}. The validity of \gr{G3} and \gr{G5} is immediate. From (A2) and (A4) we deduce that the non-negative function $\widetilde{F}'$ is increasing with finite limit on account of (A3) (recall Remark \ref{rm1.1}), hence 
\begin{align}
 \tag{A7} 0\leq \widetilde{F}'(t)\leq \nu_5\frac{1}{1+t},\;t\geq 0,
\end{align}
provided we replace $\nu_5$ from (A5) by a larger constant if necessary. Next we observe the formula
\begin{align*}
 D^2F(P)(Q,Q)=\frac{1}{|P|}\widetilde{F}'\big(|P|\big)\left[|Q|^2-\frac{(P: Q)}{|P|^2}\right]+\widetilde{F}''\big(|P|\big)\frac{(P: Q)^2}{|P|^2},\;\;P,Q\in \R^{N\times n},
\end{align*}
implying the estimate
\begin{align}\tag{A8}
\begin{split}
 \min\Big\{\widetilde{F}''\big(|P|\big),\frac{1}{|P|}\widetilde{F}'\big(|P|\big)\Big\}|Q|^2\leq D^2F(P)(Q,Q)\\
\leq \max\Big\{\widetilde{F}''\big(|P|\big),\frac{1}{|P|}\widetilde{F}'\big(|P|\big)\Big\}|Q|^2,\;\;P,Q\in\R^{N\times n}.
\end{split}
\end{align}
In conclusion, \gr{growth} follows from (A4), (A5), (A7) and (A8). In the same manner \gr{G8} is deduced from the additional hypothesis (A6). \qed


\begin{tabular}{l l}
Michael Bildhauer (bibi@math.uni-sb.de) &  Xiao Zhong (xiao.x.zhong@helsinki.fi) \\
Martin Fuchs (fuchs@math.uni-sb.de) & FI-00014 University of Helsinki \\  
Jan Müller (jmueller@math.uni-sb.de) & Department of Mathematics\\
Saarland University  & P.O. Box 68 (Gustaf Hällströmin katu 2b) \\ 
Department of Mathematics & 00100 Helsinki\\  
P.O. Box 15 11 50 & Finland\\
66041 Saarbrücken \\
Germany 

\end{tabular}

\end{document}